\documentclass[a4paper]{article}
\usepackage[pdftex]{graphicx}
\usepackage[pdftex]{color}
\usepackage{amssymb}
\usepackage{amsmath}
\usepackage{mathrsfs}
\usepackage{color}
\usepackage{amsthm}
\usepackage{multirow}
\DeclareMathAlphabet{\mathpzc}{OT1}{pzc}{m}{it}

\newtheorem{theorem}{Theorem}[section]
\newtheorem{corollary}{Corollary}

\newtheorem{lemma}[theorem]{Lemma}
\newtheorem{proposition}{Proposition}

\theoremstyle{definition}
\newtheorem{definition}[theorem]{Definition}
\newtheorem{remark}{Remark}
\newtheorem*{notation}{Notation}
\newcommand{\ep}{\varepsilon}

\newcommand{\vecf}{\mbox{\boldmath $ f $}}

\newcommand{\vecg}{\mbox{\boldmath $ g $}}
\newcommand{\vecp}{\mbox{\boldmath $ p $}}

\newcommand{\vecx}{\mbox{\boldmath $ x $}}
\newcommand{\svecx}{\mbox{\scriptsize \boldmath $ x $}}
\newcommand{\vecv}{\mbox{\boldmath $ v $}}

\newcommand{\vecphi}{\mbox{\boldmath $ \phi $}}
\newcommand{\vecpsi}{\mbox{\boldmath $ \psi $}}

\makeatletter
    
    \@addtoreset{equation}{section}
  \makeatother

\title{Recent topics on the O'Hara energies}
\author{Shoya Kawakami}
\date{\today}

\begin{document}

\maketitle

%
%

\begin{abstract}
The O'Hara energies,
introduced by Jun O'Hara in 1991,
were proposed to answer the question of what is a ``good" figure in a given knot type.
A property of the O'Hara energies is that the ``better'' the figure of a knot is, the less the energy value is.
In this article, we discuss two topics on the O'Hara energies.
First,
we slightly generalize the O'Hara energies and consider a characterization of its finiteness.
The finiteness of the O'Hara energies was considered by Blatt in 2012 who used the Sobolev-Slobodeckii space,
and naturally we consider a generalization of this space.
Another fundamental problem is to understand the minimizers of the O'Hara energies.
This problem has been addressed in several papers,
some of them based on numerical computations.
In this direction,
we discuss a discretization of the O'Hara energies and give some examples of numerical computations.
Particular one of the O'Hara energies,
called the M\"{o}bius energy thanks to its M\"{o}bius invariance,
was considered by Kim-Kusner in 1993,
and Scholtes in 2014 established convergence properties.
We apply their argument in general since the argument does not rely on M\"{o}bius invariance.
\end{abstract}

%
%

\section{Introduction}

The family of \textit{O'Hara energies} were introduced by O'Hara \cite{O91,O92} and are defined as
\[
	\mathcal{E}^{\alpha , p}(\vecf)
	:=
	\iint_{(\mathbb{R} / \mathcal{L}\mathbb{Z})^2}
	\left(
	\frac 1 { \| \vecf(s_2) - \vecf(s_1) \|_{\mathbb{R}^d}^\alpha}
	-
	\frac 1 { \mathscr{D} (\vecf(s_1) , \vecf(s_2))^\alpha}
	\right)^p
	ds_2ds_1,
\]
where
$\alpha$,
$p \in (0,\infty)$
are constants,
$\vecf : \mathbb{R} / \mathcal{L}\mathbb{Z} \to \mathbb{R}^d$
is a curve embedded in
$\mathbb{R}^d$
parametrized by arc-length with total length
$\mathcal{L}$,
and
$\mathscr{D}(\vecf(s_1) , \vecf(s_2))$
is the intrinsic distance between
$\vecf(s_1)$
and
$\vecf(s_2)$.
The purpose of these energies is to give an answer to the question:
``What is the most beautiful knot in a given knot class ?"
Therefore,
the O'Hara energies were constructed so that the more a knot is well-balanced,
the less the energy is.
Also,
when we deform a knot,
it is not desirable that the knot class to which the knot belongs changes.
Thus,
these energies were also constructed so that the energy value diverges if the curve has self-intersection.
A study of minimizers of the O'Hara energies under length constraint were carried out in
\cite{ACFGH03,FHW94,O94}.
In particular,
right circles attain the minimum of these energies for
$\alpha \in (0,\infty)$
and
$p \in [1,\infty)$
with
$0 < \alpha < 2+1/p$.
Indeed,
this result was shown by Adams et al.\ in
\cite{ACFGH03}
for the more general energy
\[
	\mathcal{E}^F (\vecf)
	:=
	\iint_{ (\mathbb{R} / \mathcal{L}\mathbb{Z})^2 }
	F( \| \vecf(s_2) - \vecf(s_1) \|_{\mathbb{R}^d} , \mathscr{D}(\vecf(s_1) , \vecf(s_2) )
	ds_2ds_1,
\]
where
$F = F(x,y)$
is increasing and convex for
$x \in (0,y]$
and
$y \in (0, \mathcal{L}/2)$.
For example,
$F(x,y) = ( x^{-\alpha} - y^{-\alpha} )^p$
satisfies this assumption when
$p \in [1,\infty)$
and
$0 < \alpha < 2+1/p$.

The purpose of this article is two-fold.
Firstly,
we study a slightly generalized energy
\[
	\mathcal{E}^{\Phi , p} (\vecf)
	:=
	\iint_{ (\mathbb{R} / \mathcal{L}\mathbb{Z})^2 }
	\left(
	\frac 1 { \Phi(\| \vecf(s_2) - \vecf(s_1) \|_{\mathbb{R}^d}) }
	-
	\frac 1 { \Phi(\mathscr{D}(\vecf(s_1) , \vecf(s_2))) }
	\right)^p
	ds_2ds_1
\]
under suitable assumptions on
$\Phi$,
and we should see that such a generalization brings out certain properties of
$\mathcal{E}^{\alpha , p}$
in a clearer manner.
It is known that the finiteness of
$\mathcal{E}^{\alpha , p}(\vecf)$
implies bi-Lipschitz continuity and some regularity of
$\vecf$,
see
\cite{B12}.
We generalize this fact to the case
$\mathcal{E}^{\Phi , p}$,
and we clarify what properties of
$\Phi$
give rise to these properties of
$\vecf$.
In particular,
we define a function space
$W^{k+\Phi , p}$
which is a generalization of the Sobolev-Slobodeckii space,
and discuss the relation between our new space and the domain of
$\mathcal{E}^{\Phi , p}$.

The second aim is to study energies for polygonal knots,
that is,
discretization of the original energy.
More precisely,
a discrete version of
$\mathcal{E}^{\alpha , p}$
is proposed together with some numerical results. 
Several discrete versions of the energy
$\mathcal{E}^{2,1}$,
called the \textit{M\"{o}bius energy},
have been introduced earlier;
one is by Kim-Kusner
\cite{KK93},
and another is by Simon
\cite{Si94}.
Their convergence was shown by Scholtes
\cite{S14}
and Rawdon-Simon
\cite{RS06}
respectively.
Although
$\mathcal{E}^{2,1}$
is invariant under M\"{o}bius transformations
(cf.\ \cite{FHW94}),
the proof of the result of
\cite{S14}
did not use the M\"{o}bius invariance.
Here,
we extend the results by Kim-Kusner
\cite{KK93}
and Scholtes
\cite{S14}
to the case
$\mathcal{E}^{\alpha ,p}$,
and improve the rate of convergence of
$\mathcal{E}^{2,1}$.

%
%

\section{A generalization of the O'Hara energy}\label{gene}

Although minimizers of
$\mathcal{E}^F$
were obtained in
\cite{ACFGH03},
other fundamental properties of
$\mathcal{E}^F$
have not been investigated in the existing literature.
Here,
we consider the problem of characterizing the finiteness of generalized energies.
At the level of generality of
$\mathcal{E}^F$,
this seems to be a very difficult problem so we restrict ourself to the case
$\mathcal{E}^{\Phi ,p}$
defined above,
where
$p \in [1,\infty)$
is a constant,
and
$\Phi : [0,\infty) \to [0,\infty)$
is a strictly increasing function such that
$\Phi(0) = 0$.
Note that finiteness of the O'Hara energies is discussed by Blatt in \cite{B12},
where he showed that if
$\alpha \in (0,\infty)$
and
$q \in [1,\infty)$
satisfy
$2 \leq \alpha p < 2p+1$,
then
$\mathcal{E}^{\alpha ,p}(\vecf) < \infty$
if and only if
$\vecf$
is bi-Lipschitz continuous and belongs to the Sobolev-Slobodeckii space
\begin{multline*}
	W^{1+\sigma , 2p} (\mathbb{R} / \mathcal{L}\mathbb{Z} , \mathbb{R}^d)
\\
	:=
	\left\{ \vecf \in W^{1,2p} (\mathbb{R} / \mathcal{L}\mathbb{Z} , \mathbb{R}^d)
	\,\left|\,
	\iint_{(\mathbb{R} / \mathcal{L}\mathbb{Z})^2}
	\frac{ \| \vecf^\prime(s_2)-\vecf^\prime(s_1) \|_{\mathbb{R}^d}^{2p} }{ |s_2-s_1|^{\alpha p} }
	ds_2ds_1 < \infty
	\right.\right\},
\end{multline*}
where
$\sigma = (\alpha p-1)/(2p)$.
Hence,
to establish a condition for finiteness of
$\mathcal{E}^{\Phi , p}$,
it is natural to consider a generalization of the Sobolev-Slobodeckii space.

\begin{definition}
Let
$\Omega$
be a non-empty subset of
$\mathbb{R}$.
For
$p \in [1, \infty)$,
$k \in \mathbb{N} \cup \{ 0 \}$,
and measurable function
$\Psi : [0,\infty) \to [0,\infty)$,
we define
\[
	W^{k+\Psi , p}(\Omega , \mathbb{R}^d)
	:=
	\{ \vecf \in W^{k,p}(\Omega , \mathbb{R}^d)
	\,|\,
	[\vecf^{(k)}]_{\Psi , p} < \infty \},
\]
where
\[
	[\vecf^{(k)}]_{\Psi , p}
	:=
	\left(
	\iint_{\Omega \times \Omega}	
	\frac{ \| \vecf^{(k)}(s_2) - \vecf^{(k)}(s_1) \|_{\mathbb{R}^d}^p }{ \Psi(|s_2-s_1|)^p }
	\frac 1 {|s_2-s_1|} ds_2ds_1
	\right)^{1/p}.
\]
\end{definition}
\noindent
We equip the space
$W^{k+\Psi ,p}$
with the norm
\[
	\| \vecf \|_{W^{k+\Psi , p}}
	:=
	\| \vecf \|_{W^{k ,p}} + [\vecf^{(k)}]_{\Psi , p},
\]
in which case it becomes a Banach space.
Moreover,
the dual space of
$W^{\Psi , p} (\Omega , \mathbb{R}^d)$
is characterized as the following proposition which may be proved by using the argument of
\cite[pp.\ 38--42]{M06}.

\begin{proposition}
Let
$\Omega$
be a non-empty subset of
$\mathbb{R}$,
and let
$\Psi : [0,\infty) \to [0,\infty)$
be a measurable function.
For
$p \in [1,\infty)$,
let
$q \in (1,\infty]$
satisfy
$1/p + 1/q = 1$.
Then,
for all
$T \in (W^{\Psi , p} (\Omega , \mathbb{R}^d) )^\prime$,
there exists
$(\vecphi , \vecpsi) \in L^q (\Omega , \mathbb{R}^d) \times L^q (\Omega\times\Omega , \mathbb{R}^d)$
such that
\[
	\| T \|_{(W^{\Psi , p}(\Omega))^\prime}
	=
	\max \{ \| \vecphi \|_{L^q (\Omega)} , \| \vecpsi \|_{L^q (\Omega\times\Omega)} \}
\]
and
\[
	T(\vecf)
	=
	\int_\Omega \vecf(s) \cdot \vecphi(s) ds
	+
	\iint_{\Omega \times \Omega}
	\left(
	\frac{ \vecf(s_2) - \vecf(s_1) }{ \Psi(|s_2-s_1|) } \cdot \vecpsi(s_1 , s_2)
	\right)
	\frac 1 { |s_2-s_1|^{1/p} }
	ds_2ds_1
\]
for any
$\vecf \in W^{\Psi , p} (\Omega , \mathbb{R}^d)$.
In particular,
if
$1 < p < \infty$,
then
$W^{\Psi , p} (\Omega , \mathbb{R}^d)$
is reflexive.
\end{proposition}

In
\cite{B12},
it was shown that
$\vecf$
is \textit{bi-Lipschitz continuous} for all embedded regular curves
$
	\vecf \in C^{0,1} (\mathbb{R} / \mathcal{L}\mathbb{Z} , \mathbb{R}^d)
	\cap
	W^{1+\Psi , 2p} (\mathbb{R} / \mathcal{L}\mathbb{Z} , \mathbb{R}^d)
$,
which suggests that
$\vecf$
does not bend sharply.
It is natural to expect that all embedded regular curves belonging to the generalized Sobolev space are bi-Lipschitz;
we confirm this expectation with the following theorem which we establish by modifying the argument of Blatt
\cite{B12}.

\begin{theorem}[The bi-Lipschitz continuity]\label{thm:bi}
Let an increasing function
$\Phi : [0,\infty) \to [0,\infty)$
satisfy
$\Phi(0) = 0$
and
$\Phi(x) = O(x^{2/p})$
as
$x \to +0$
for
$p \in [1,\infty)$.
Set
$\Psi(x) := ( x^{-1/p} \Phi(x) )^{1/2}$.
Assume that
$\vecf$
belongs to
$
	C^{0,1} (\mathbb{R} / \mathcal{L}\mathbb{Z} , \mathbb{R}^d)
	\cap
	W^{1+\Psi , 2p} (\mathbb{R} / \mathcal{L}\mathbb{Z} , \mathbb{R}^d)
$
whose image is a closed embedded curve in
$\mathbb{R}^d$
parametrized by arc-length.
Then,
$\vecf$
is bi-Lipschitz continuous.
\end{theorem}

\begin{proof}
We only have to prove that there exists
$C_{\text b} > 0$
such that
\[
	\| \vecf(s_2) - \vecf(s_1) \|_{\mathbb{R}^d} \geq C_{\text b} \mathscr{D}(\vecf(s_1) , \vecf(s_2))
\]
for
$s_1 ,\,s_2 \in \mathbb{R} / \mathcal{L}\mathbb{Z}$.

Note that there exists
$M$,
$\delta > 0$
such that if
$x< \delta$,
then we have
$\Phi(x) \leq M x^{2/p}$
because
$\Phi(x) = O(x^{2/p})$
as
$x \to +0$.
By the assumption
$\vecf \in W^{1+\Psi , 2p} (\mathbb{R} / \mathcal{L}\mathbb{Z} , \mathbb{R}^d)$,
we have
\[
	[ \vecf^\prime ]_{\Psi , 2p}^{2p}
	=
	\int_{\mathbb{R} / \mathcal{L}\mathbb{Z}} \int_{-\mathcal{L}/2}^{\mathcal{L}/2}
	\frac{ \| \vecf^\prime(s_1+s_2) - \vecf^\prime(s_1) \|_{\mathbb{R}^d}^{2p} }{ \Phi(|s_2|)^p }
	ds_2ds_1
	< \infty.
\]
Using Lebesgue's dominated convergence theorem,
we have
\[
	\lim_{r \to +0}
	\int_{\mathbb{R} / \mathcal{L}\mathbb{Z}} \int_{-r}^r
	\frac{ \| \vecf^\prime(s_1+s_2) - \vecf^\prime(s_1) \|_{\mathbb{R}^d}^{2p} }{ \Phi(|s_2|)^p }
	ds_2ds_1
	= 0.
\]
Therefore,
there exists
$\eta \in (0, \min\{ \delta , 1 , \mathcal{L} \}/2 )$
such that
\[
	\sup_{s \in \mathbb{R} / \mathcal{L}\mathbb{Z}}
	\int_{s-r}^{s+r} \int_{-r}^r
	\frac{ \| \vecf^\prime(s_1+s_2) - \vecf^\prime(s_1) \|_{\mathbb{R}^d}^{2p} }{ \Phi(|s_2|)^p }
	ds_2ds_1
	\leq \frac 1 {2^{2p} M^p}
\]
if
$r \leq \eta$.
Hence,
for
$s \in \mathbb{R} / \mathcal{L}\mathbb{Z}$
and
$r \leq \eta$,
we get
\begin{align*}
	&\frac 1 {2r} \int_{s-r}^{s+r}
	\left\|
	\vecf^\prime(s_1) - 
	\frac 1 {2r} \int_{s-r}^{s+r} \vecf^\prime (s_2) ds_2
	\right\|_{\mathbb{R}^d} ds_1
\\
	&\leq
	\frac 1 {4r^2} \int_{s-r}^{s+r} \int_{s-r}^{s+r}
	\| \vecf^\prime(s_2) - \vecf^\prime(s_1) \|_{\mathbb{R}^d}
	ds_2ds_1
\\
	&\leq
	\left(
	\frac{ \Phi(2r)^p }{4r^2} \int_{s-r}^{s+r} \int_{s-r}^{s+r}
	\frac{ \| \vecf^\prime(s_2) - \vecf^\prime(s_1) \|_{\mathbb{R}^d}^{2p} }{ \Phi(|s_2-s_1|)^p }
	ds_2ds_1
	\right)^{1/2p}
	\leq \frac 1 2.
\end{align*}
Now,
let
$s_1$,
$s_2$,
$s_3 \in \mathbb{R} / \mathcal{L}\mathbb{Z}$
with
$|s_2-s_1| = 2r$
$(\leq 2 \eta)$
and
$\mathscr{D} (\vecf(s_1) , \vecf(s_3)) = \mathscr{D} (\vecf(s_3) , \vecf(s_2))$.
Then,
we have
\begin{align*}
	\| \vecf(s_2) - \vecf(s_1) \|_{\mathbb{R}^d}
	&=
	\sup_{\| \svecx \|_{\mathbb{R}^d} \leq 1}
	\int_{s_3-r}^{s_3+r} \vecf^\prime(s) \cdot \vecx ds
\\
	&=
	2r +
	\sup_{\| \svecx \|_{\mathbb{R}^d} \leq 1}
	\int_{s_3-r}^{s_3+r} \vecf^\prime(s) \cdot (\vecf^\prime(s) - \vecx ) ds
\\
	&\geq
	2r -
	\inf_{\| \svecx \|_{\mathbb{R}^d} \leq 1}
	\int_{s_3-r}^{s_3+r} \| \vecf^\prime(s)  - \vecx \|_{\mathbb{R}^d} ds
\\
	&=
	\left(
	1-
	\inf_{\| \svecx \|_{\mathbb{R}^d} \leq 1}
	\frac 1 {2r}
	\int_{s_3-r}^{s_3+r} \| \vecf^\prime(s)  - \vecx \|_{\mathbb{R}^d} ds
	\right) |s_2-s_1|
\\
	&\geq \frac 1 2 |s_2-s_1|
\end{align*}
because
\[
	\left\| \frac 1 {2r} \int_{s_3-r}^{s_3+r} \vecf(s) ds \right\|_{\mathbb{R}^d} \leq 1.
\]

Next,
we consider the case where
$\mathscr{D}(\vecf(s_1) , \vecf(s_2)) \geq 2 \eta$.
Let
\[
	I_\eta :=
	\{ (s_1,s_2) \in (\mathbb{R} / \mathcal{L}\mathbb{Z})^2
	\,|\,
	\mathscr{D}(\vecf(s_1) , \vecf(s_2)) \geq 2\eta \}.
\]
Then,
we have
\[
	C_{\text b}^\eta := \inf_{(s_1 , s_2) \in I_\eta}
	\frac{ \| \vecf(s_2) - \vecf(s_1) \|_{\mathbb{R}^d} }{ \mathscr{D}(\vecf(s_1) , \vecf(s_2)) }
	>0
\]
because
$\vecf$
has no self-intersection.
Therefore,
we obtain
\[
	\| \vecf(s_2) - \vecf(s_1) \|_{\mathbb{R}^d} \geq C_{\text b}^\eta \mathscr{D}(\vecf(s_1) , \vecf(s_2)).
\]
\end{proof}

Using the space
$W^{k+\Psi , 2p}$,
we establish the following theorem concerning the finiteness of the energies
$\mathcal{E}^{\Phi , p}$.

\begin{theorem}[Finiteness of $\mathcal{E}^{\Phi , p}(\vecf)$]\label{thm:bdd}
Let
$p \in [1,\infty)$,
and let
$\vecf \in C^{0,1} (\mathbb{R} / \mathcal{L}\mathbb{Z} , \mathbb{R}^d)$
be a function whose image is a closed curve parametrized by arc-length embedded in
$\mathbb{R}^d$
with total length
$\mathcal{L}$.
Assume that a measurable function
$\Phi : [0,\infty) \to [0,\infty)$
satisfies the following.
\begin{itemize}
\item[{\rm (A0)}]
	$\Phi(0) = 0$,
	$\Phi \in C^1$,
	and
	$\Phi^\prime (x) > 0$
	for
	$x > 0$.
\item[{\rm (A1)}]
	There exists
	$K>0$
	such that
	$\displaystyle\lim_{x \to +0}G(x) = K$,
	where
	$\displaystyle G(x) := \frac{x \Phi^\prime(x)}{\Phi(x)}$.
\item[{\rm (A2)}]
	There exists a measurable function
	$\varphi : [0,\infty) \to [0,\infty)$
	such that
	\begin{itemize}
	\item[{\rm (A2-1)}]
		$\Phi(kx) \leq \varphi(k)\Phi(x)$
		for
		$k,\,x \geq 0$,
	\end{itemize}
	and
	$\displaystyle M(a) := \int_0^a \frac{\varphi(t)^p} t dt$
	($a>0$)
	satisfies
	\begin{itemize}
	\item[{\rm (A2-2)}]
		$M(\ep) = o(\ep)$
		as
		$\ep \to +0$,
	\item[{\rm (A2-3)}]
		$M(a) < \infty$
		for
		$a>0$.
	\end{itemize}
\item[{\rm (A3)}]
	$\displaystyle \int_0^a \frac{t^{2p}}{\Phi(t)^p} dt < \infty$
	for
	$a >0$.
\end{itemize}
Set
$\displaystyle \Psi(x) := \left( \frac{\Phi(x)}{x^{1/p}} \right)^{1/2}$
for
$x>0$.
Then,
we have the following two properties.
\begin{itemize}
\item[{\rm 1.}]
	If
	$\vecf \in W^{1+\Psi , 2p} (\mathbb{R} / \mathcal{L}\mathbb{Z} , \mathbb{R}^d)$
	and
	$\vecf$
	is bi-Lipschitz continuous,
	then we have
	$\mathcal{E}^{\Phi, p}(\vecf) < \infty$.
\item[{\rm 2.}]
	If
	$\mathcal{E}^{\Phi, p}(\vecf) < \infty$,
	then
	$\vecf$
	belongs to
	$W^{1+\Psi , 2p} (\mathbb{R} / \mathcal{L}\mathbb{Z} , \mathbb{R}^d)$.
	
	Moreover,
	there exists
	$C>0$
	depending only
	$p$,
	$\mathcal{L}$,
	and
	$\Phi$
	such that
	\begin{equation}
		\| \vecf^\prime \|_{W^{\Psi , 2p}}^{2p} \leq C( \mathcal{E}^{\Phi , p}(\vecf) + \| \vecf^\prime \|_{L^{2p}}).
	\label{thm:bdd:ineq}
	\end{equation}
\end{itemize}
\end{theorem}

\begin{remark}\label{rmk:thm:bdd}
Suppose we assume
\begin{itemize}
\item[(A2-2)${}^\prime$]
	$\varphi(x) = O(x^{2/p})$
	as
	$x \to \infty$
\end{itemize}
instead of (A2-2) in Theorem \ref{thm:bdd}.
Then,
we have
$M(\ep) = o(\ep)$
as
$\ep \to +0$,
and using the argument of
\cite{O94},
we can prove that
$\vecf$
is bi-Lipschitz continuous if
$\mathcal{E}^{\Phi , p}(\vecf) < \infty$.
Thus,
it holds that
$\mathcal{E}^{\Phi , p}(\vecf) < \infty$
if and only if
$
	\vecf \in W^{1+\Psi , 2p} (\mathbb{R} / \mathcal{L}\mathbb{Z} , \mathbb{R}^d)
	\cap
	C^{0,1} (\mathbb{R} / \mathcal{L}\mathbb{Z} , \mathbb{R}^d)
$
and
$\vecf$
is bi-Lipschitz continuous.
\end{remark}

The following table shows ranges of
$\alpha$
satisfying the assumptions of Theorems \ref{thm:bi} or \ref{thm:bdd},
which contains some examples of
$\Phi$.
The column ``Remark \ref{rmk:thm:bdd}'' shows ranges of $\alpha$
satisfying (A0), (A1), (A2-1), (A2-2)${}^\prime$, (A2-3), and (A3).
\begin{table}[h]
\small
\begin{center}
{\renewcommand\arraystretch{1.1}
\begin{tabular}{r||c|c|c}
 & $\Phi(x) = x^\alpha$ & $\Phi(x) = x^\alpha\log (x+1)$ & $\Phi(x) = 1-e^{-x^\alpha}+x^{2\alpha} /2$\\ \hline \hline
Theorem \ref{thm:bi} & $[2/p , \infty)$ & $[2/p-1 , \infty)$ & $[1/p, \infty)$\\ \hline
Theorem \ref{thm:bdd} & $(1/p , 2+1/p)$ & $(1/p , 1/p+1)$ & $(1/p , 2+1/p)$\\ \hline
Remark \ref{rmk:thm:bdd} & $[2/p , 2+1/p)$ & $[2/p,1/p+1)$ & $[2/p , 2+1/p)$\\
 & & $(p>1)$ &
\end{tabular}
}
\normalsize
\caption{Examples of $\Phi$}
\end{center}
\end{table}

\begin{notation}
For
$s_1$,
$s_2 \in \mathbb{R} / \mathcal{L}\mathbb{Z}$
and
$\vecv : \mathbb{R} / \mathcal{L}\mathbb{Z} \to \mathbb{R}^d$,
we write
$\Delta_{s_1}^{s_2} \vecv := \vecv(s_2) - \vecv(s_1)$.
\end{notation}

The proof of Theorem \ref{thm:bdd} is based on an argument by Blatt
\cite{B12}.
Before proving Theorem \ref{thm:bdd},
we establish the following lemma which is used in proof of inequality (\ref{thm:bdd:ineq}).
Let
\[
	\tilde{\mathcal{E}}^{\Phi , p} (\vecg)
	:=
	\int_{\mathbb{R} / \mathcal{L}\mathbb{Z}} \int_{-\mathcal{L}/2}^{\mathcal{L}/2}
	\frac{ \left( \int_0^1 \int_0^1 \| \Delta_{s_1+s_3s_2}^{s_1+s_4s_2} \vecg \|_{\mathbb{R}^d} ds_4ds_3 \right)^p }
	{ \Phi(|s_2|)^p }
	ds_2ds_1
\]
for
$\vecg : \mathbb{R} / \mathcal{L}\mathbb{Z} \to \mathbb{R}^d$.

\begin{lemma}\label{lem:bdd}
There exists
$C = C(p, \mathcal{L}, \Phi)>0$
such that
\[
	[\vecg]_{\Psi , 2p}^{2p} \leq C \left( \tilde{\mathcal{E}}^{\Phi , p} (\vecg) + \| \vecg \|_{L^{2p}}^{2p} \right)
\]
for all almost-everywhere continuous functions
$\vecg : \mathbb{R} / \mathcal{L}\mathbb{Z} \to \mathbb{R}^d$.
\end{lemma}

\begin{proof}
First,
we consider the case where
$\vecg \in C^\infty (\mathbb{R} / \mathcal{L}\mathbb{Z} , \mathbb{R}^d)$.
For
$\ep \in (0,1)$,
we decompose
\[
	[\vecg]_{\Psi , 2p}^{2p}
	=
	J_\ep^1(\vecg) + J_\ep^2(\vecg),
\]
where
\begin{align*}
	J_\ep^1(\vecg)
	&:=
	\int_{\mathbb{R} / \mathcal{L}\mathbb{Z}}
	\int_{|s_2| \geq \ep\mathcal{L}/2}
	\frac{ \| \Delta_{s_1}^{s_1+s_2} \vecg \|_{\mathbb{R}^d}^{2p} }{ \Phi(|s_2|)^p }
	ds_2ds_1,
\\
	J_\ep^2(\vecg)
	&:=
	\int_{\mathbb{R} / \mathcal{L}\mathbb{Z}}
	\int_{|s_2| \leq \ep\mathcal{L}/2}
	\frac{ \| \Delta_{s_1}^{s_1+s_2} \vecg \|_{\mathbb{R}^d}^{2p} }{ \Phi(|s_2|)^p }
	ds_2ds_1.
\end{align*}
Now,
we have
\[
	J_\ep^1(\vecg)
	\leq
	\frac{ 2^{2p}\mathcal{L} }{ \Phi(\ep\mathcal{L}) }
	\| \vecg \|_{L^{2p}}^{2p}
\]
because
$\Phi$
is an increasing function.
As in
\cite{B12},
it is not difficult to see
\[
	\ep^{2p} J_\ep^2(\vecg)
	\leq
	3^p( K_\ep^1(\vecg) + K_\ep^2(\vecg) + K_\ep^3(\vecg) ),
\]
where
\begin{align*}
	K_\ep^1(\vecg)
	&:=
	\int_{\mathbb{R} / \mathcal{L}\mathbb{Z}}
	\int_{|s_2| \leq \ep\mathcal{L}/2}
	\frac{ \left( \int_0^\ep \int_{1-\ep}^1
	\| \Delta_{s_1+s_4s_2}^{s_1+s_2} \vecg \|_{\mathbb{R}^d}^2 ds4ds_3 \right)^p}
	{ \Phi(|s_2|)^p }
	ds_2ds_1,
\\
	K_\ep^2(\vecg)
	&:=
	\int_{\mathbb{R} / \mathcal{L}\mathbb{Z}}
	\int_{|s_2| \leq \ep\mathcal{L}/2}
	\frac{ \left( \int_0^\ep \int_{1-\ep}^1
	\| \Delta_{s_1+s_3s_2}^{s_1+s_4s_2} \vecg \|_{\mathbb{R}^d}^2 ds4ds_3 \right)^p}
	{ \Phi(|s_2|)^p }
	ds_2ds_1,
\\
	K_\ep^3(\vecg)
	&:=
	\int_{\mathbb{R} / \mathcal{L}\mathbb{Z}}
	\int_{|s_2| \leq \ep\mathcal{L}/2}
	\frac{ \left( \int_0^\ep \int_{1-\ep}^1
	\| \Delta_{s_1}^{s_1+s_3s_2} \vecg \|_{\mathbb{R}^d}^2 ds4ds_3 \right)^p}
	{ \Phi(|s_2|)^p }
	ds_2ds_1.
\end{align*}
By the definition of
$\tilde{\mathcal{E}}^{\Phi , p} (\vecg)$,
we have
$K_\ep^2(\vecg) \leq \tilde{\mathcal{E}}^{\Phi , p}(\vecg)$.
Moreover,
we have
\begin{align*}
	&K_\ep^3(\vecg)
	=
	\ep^p
	\int_{\mathbb{R} / \mathcal{L}\mathbb{Z}}
	\int_{|s_2| \leq \ep\mathcal{L}/2}
	\frac{ \left( \int_0^\ep
	\| \Delta_{s_1}^{s_1+s_3s_2} \vecg \|_{\mathbb{R}^d}^{2p} ds_3 \right)^p}
	{ \Phi(|s_2|)^p }
	ds_2ds_1
\\
	&\leq
	\ep^{2p-1}
	\int_{\mathbb{R} / \mathcal{L}\mathbb{Z}}
	\int_0^\ep
	\int_{|s_2| \leq \ep\mathcal{L}/2}
	\frac{ \| \Delta_{s_1}^{s_1+s_3s_2} \vecg \|_{\mathbb{R}^d}^{2p} }
	{ \Phi(|s_2|)^p }
	ds_2ds_3ds_1
\\
	&=
	\ep^{2p-1}
	\int_{\mathbb{R} / \mathcal{L}\mathbb{Z}}
	\int_0^\ep
	\int_{|s_2| \leq s_3\ep\mathcal{L}/2}
	\frac{ \| \Delta_{s_1}^{s_1+\tilde{s}_2} \vecg \|_{\mathbb{R}^d}^{2p} }
	{ \Phi(|\tilde{s}_2| / s_3)^p s_3}
	d\tilde{s}_2ds_3ds_1
\\
	&\leq
	\ep^{2p-1}
	\int_{\mathbb{R} / \mathcal{L}\mathbb{Z}}
	\int_0^\ep
	\int_{|s_2| \leq s_3\ep\mathcal{L}/2}
	\frac{ \varphi(s_3)^p }{ s_3 }
	\frac{ \| \Delta_{s_1}^{s_1+\tilde{s}_2} \vecg \|_{\mathbb{R}^d}^{2p} }
	{ \Phi(|\tilde{s}_2|)^p }
	d\tilde{s}_2ds_3ds_1
\\
	&\leq
	M(\ep) \ep^{2p-1} J_\ep^2(\vecg)
\end{align*}
by H\"{o}lder's inequality,
Fubini's Theorem,
and (A2-1).
Also,
we have
$K_\ep^1(\vecg) = K_\ep^3(\vecg)$
by a change of variables.
Hence,
we get
\[
	J_\ep^2(\vecg) \leq
	\frac{ 3^p }{ \ep^{2p} } \tilde{\mathcal{E}}^{\Phi , p} (\vecg)
	+
	3^p \cdot 2\frac{M(\ep)} \ep J_\ep^2(\vecg).
\]
Now,
we can take
$\ep$
sufficiently small satisfying
\[
	3^p \cdot 2 \frac{M(\ep)} \ep < 1
\]
by (A2-2).
Then,
we get
\[
	J_\ep^2 (\vecg) \leq C(p , \ep , \varphi) \tilde{\mathcal{E}}^{\Phi , p} (\vecg)
\]
because
$J_\ep^2 (\vecg) < \infty$
by
$\vecg \in C^\infty (\mathbb{R} / \mathcal{L}\mathbb{Z} , \mathbb{R}^d)$
and because of (A3),
where
$C(p , \ep , \varphi)$
is a positive constant.
Therefore,
we obtain
\[
	[\vecg]_{\Psi , 2p}^{2p}
	\leq
	\frac{ 2^{2p}\mathcal{L} }{ \Phi(\ep\mathcal{L}/2) } \| \vecg \|_{L^{2p}}^{2p}
	+
	C(p , \ep , \varphi) \tilde{\mathcal{E}}^{\Phi , p}(\vecg).
\]

Next,
we consider the case where
$\vecg$
is an almost everywhere continuous function.
Let
$\phi \in C^\infty_0 (\mathbb{R})$
with
$\operatorname{supp} \phi \subset [-\mathcal{L}/2 , \mathcal{L}/2]$
and
\[
	\int_{-\mathcal{L}/2}^{\mathcal{L}/2} \phi(x) dx = 1,
\]
and define
$\phi_\ep (x) := \ep^{-1} \phi(x/\ep)$
for
$x \in \mathbb{R}$.
Set
\[
	\vecg_\ep (s) := \int_{-\mathcal{L}/2}^{\mathcal{L}/2} \phi_\ep(s) \vecg(s+x)ds.
\]
Then,
we have
\[
	[\vecg_\ep]_{\Psi , 2p}^{2p}
	\leq
	\frac{ 2^{2p}\mathcal{L} }{ \Phi(\ep\mathcal{L}/2) } \| \vecg_\ep \|_{L^{2p}}^{2p}
	+
	C(p , \ep , \varphi) \tilde{\mathcal{E}}^{\Phi , p}(\vecg_\ep)
\]
because
$\vecg_\ep \in C^\infty$.
Also,
we have
\[
	\| \vecg_\ep \|_{L^{2p}} = \| \vecg  \|_{L^{2p}},
	\qquad
	\tilde{\mathcal{E}}^{\Phi , p} (\vecg_\ep) \leq \tilde{\mathcal{E}}^{\Phi , p} (\vecg).
\]
Hence,
we get
\[
	\| \vecg_\ep \|_{W^{\Psi , 2p}}^{2p}
	\leq
	2^{2p-1}
	\left\{
	\left( 1+ \frac{ 2^{2p}\mathcal{L} }{ \Phi(\ep\mathcal{L}/2) } \right)
	\| \vecg \|_{L^{2p}}^{2p}
	+
	C(p , \ep , \varphi) \tilde{\mathcal{E}}^{\Phi , p}(\vecg)
	\right\},
\]
and we can see
$\{ \vecg_\ep \}_{\ep > 0}$
is a
$W^{\Psi , 2p}$-bounded sequence.
By reflexivity of
$W^{\Psi , 2p}$,
there exists a subsequence
$\{ \vecg_{\ep_j} \}_{j=0}^{\infty}$
such that
\[
	\vecg_{\ep_j} \rightharpoonup \vecg
\]
as
$j \to \infty$.
Therefore,
we obtain
\begin{align*}
	[\vecg]_{\Psi . 2p}^{2p}
	&\leq \| \vecg \|_{W^{\Psi , 2p}}^{2p}
	\leq
	\liminf_{j \to \infty}
	\| \vecg_{\ep_j} \|_{W^{\Psi , 2p}}^{2p}
\\
	&\leq
	2^{2p-1}
	\left\{
	\left( 1+ \frac{ 2^{2p}\mathcal{L} }{ \Phi(\ep\mathcal{L}/2) } \right)
	\| \vecg \|_{L^{2p}}^{2p}
	+
	C(p , \ep , \varphi) \tilde{\mathcal{E}}^{\Phi , p}(\vecg)
	\right\}
\end{align*}
by lower semi-continuity of a weakly convergent sequence.
\end{proof}

\begin{proof}[Proof of Theorem \ref{thm:bdd}.]
For
$\ep \in (0, \mathcal{L}/2)$,
let
\[
	\mathcal{E}^{\Phi , p}_\ep (\vecf)
	:=
	\int_{\mathbb{R} / \mathcal{L}\mathbb{Z}} \int_{ \ep \leq |s_2| \leq \mathcal{L}/2 }
	( g_{|s_2|} ( \| \Delta_{s_1}^{s_1+s_2} \vecf \|_{\mathbb{R}^d }) )^p
	ds_2ds_1.
\]
Then,
we have
$\mathcal{E}^{\Phi , p} (\vecf) = \lim_{\ep \to +0} \mathcal{E}^{\Phi , p}_\ep (\vecf)$.
By the mean value theorem,
for
$s_1 \in \mathbb{R} / \mathcal{L}\mathbb{Z}$,
$\ep \leq |s_2| \leq \mathcal{L}/2$,
there exists
$\theta = \theta(s_1,s_2) \in ( \| \Delta_{s_1}^{s_1+s_2} \vecf \|_{\mathbb{R}^d} , |s_2| )$
such that
\begin{equation}
	g_{|s_2|} (\| \Delta_{s_1}^{s_1+s_2} \vecf \|_{\mathbb{R}^d})
	=
	-g_{|s_2|}^\prime(\theta) (|s_2| - \| \Delta_{s_1}^{s_1+s_2} \vecf \|_{\mathbb{R}^d}).
\label{pf:thm:bdd:mean}
\end{equation}
By (A1),
for all
$\eta > 0$,
there exists
$\delta > 0$
such that if
$0<x<\delta$
then we have
\[
	K-\eta \leq G(x) \leq K+\eta.
\]

First,
we assume that
$\vecf$
is bi-Lipschitz continuous and belongs to
$W^{1+\Psi , 2p} (\mathbb{R} / \mathcal{L}\mathbb{Z} , \mathbb{R}^d)$.
By H\"{o}lder's inequality and
(\ref{pf:thm:bdd:mean}),
we have
\begin{align*}
	&\mathcal{E}^{\Phi , p}_\ep(\vecf)
	=
	\int_{\mathbb{R} / \mathcal{L}\mathbb{Z}} \int_{\ep \leq |s_2| \leq \mathcal{L}/2}
	( g_{|s_2|} (\| \Delta_{s_1}^{s_1+s_2} \vecf \|_{\mathbb{R}^d} )^p ds_2ds_1
\\
	&=
	\int_{\mathbb{R} / \mathcal{L}\mathbb{Z}} \int_{\ep \leq |s_2| \leq \mathcal{L}/2}
	\{ -g_{|s_2|}^\prime (\theta)
	(|s_2| - \| \Delta_{s_1}^{s_1+s_2} \vecf \|_{\mathbb{R}^d})
	\}^p ds_2ds_1
\\
	&\leq
	\frac 1 {2^p}
	\int_{\mathbb{R} / \mathcal{L}\mathbb{Z}} \int_{\ep \leq |s_2| \leq \mathcal{L}/2}
	\int_0^1 \int_0^1
	( -g_{|s_2|}^\prime (\theta) |s_2| )^p
	\| \Delta_{s_1+s_4s_2}^{s_1+s_3s_2} \vecf^\prime \|_{\mathbb{R}^d}^{2p}
	ds_4ds_3ds_2ds_1.
\end{align*}
By the bi-Lipschitz continuity of
$\vecf$
and
(A2-1),
we have
\[
	-g_{|s_2|}^\prime (\theta) |s_2| \Phi(|s_4-s_3||s_2|)
	\leq
	C_{\text b} G(\theta) \varphi(C_{\text b} |s_4-s_3|)
\]
for
$s_1 \in \mathbb{R} / \mathcal{L}\mathbb{Z}$,
$\ep \leq |s_2| \leq \mathcal{L}/2$,
$s_3$,
$s_4 \in [0,1]$,
where
$C_{\text b} > 0$
is the bi-Lipschitz constant of
$\vecf$.
By (A2-3) and Fubini's theorem,
we have
\begin{align*}
	&\int_{\mathbb{R} / \mathcal{L}\mathbb{Z}} \int_{\ep \leq |s_2| \leq \mathcal{L}/2}
	\int_0^1 \int_0^1
	( -g_{|s_2|}^\prime (\theta) |s_2| )^p
	\| \Delta_{s_1+s_4s_2}^{s_1+s_3s_2} \vecf^\prime \|_{\mathbb{R}^d}^{2p}
	ds_4ds_3ds_2ds_1
\\
	&\leq
	C_{\text b}^p
	\int_{\mathbb{R} / \mathcal{L}\mathbb{Z}} \int_{-\mathcal{L}/2}^{\mathcal{L}/2}
	\int_0^1 \int_0^1 G(\theta)^p \varphi(C_{\text b} |s_4-s_3|)^p
	\frac{ \| \Delta_{s_1+s_3s_2}^{s_1+s_4s_2} \vecf^\prime \|_{\mathbb{R}^d}^{2p} }
	{ \Phi(|s_4-s_3||s_2|)^p }
	ds_4ds_3ds_2ds_1
\\
	&=
	C_{\text b}^p
	\int_0^1\int_0^1 \varphi(C_{\text b} |s_4-s_3|)^p
	\int_{-\mathcal{L}/2}^{\mathcal{L}/2} \int_{\mathbb{R} / \mathcal{L}\mathbb{Z}}
	G(\theta)^p
	\frac{ \| \Delta_{s_1+s_3s_2}^{s_1+s_4s_2} \vecf^\prime \|_{\mathbb{R}^d}^{2p} }
	{ \Phi(|s_4-s_3||s_2|)^p }
	ds_1ds_2ds_3ds_4
\\
	&=
	C_{\text b}^p
	\int_0^1\int_0^1 \frac{ \varphi(C_{\text b} |s_4-s_3|)^p }{ |s_4-s_3| }
\\
	& \quad \times
	\int_{|s_4-s_3|\ep \leq |t_2| \leq |s_4-s_3|\mathcal{L}/2}
	\int_{\mathbb{R} / \mathcal{L}\mathbb{Z}}
	G(\tilde{\theta})^p
	\frac{ \| \Delta_{t_1}^{t_1+t_2} \vecf^\prime \|_{\mathbb{R}^d}^{2p} }
	{ \Psi(|t_2|)^p }
	\frac 1 {|t_2|}
	dt_1dt_2ds_3ds_4,
\end{align*}
where
$s_1$,
$s_2$
are transformed into
$t_1 = s_1+s_3s_2$,
$t_2 = (s_4-s_3)s_2$,
and we set
$\tilde{\theta} = \tilde{\theta}(t_1,t_2) = \theta(s_1,s_2)$
in the last equality.
We take
$\ep >0$
satisfying
$\ep \leq \delta$.
For
$s_3$,
$s_4 \in [0,1]$,
we decompose
\[
	\int_{|s_4-s_3|\ep \leq |t_2| \leq |s_4-s_3|\mathcal{L}/2}
	\int_{\mathbb{R} / \mathcal{L}\mathbb{Z}}
	G(\tilde{\theta})^p 
	\frac{ \| \Delta_{t_1}^{t_1+t_2} \vecf^\prime \|_{\mathbb{R}^d}^{2p} }
	{ \Psi(|t_2|)^p }
	\frac 1 {|t_2|}
	dt_1dt_2
	= I^1_{\ep,\delta}(\vecf^\prime) + I^2_{\ep,\delta}(\vecf^\prime),
\]
where
\begin{align*}
	I^1_{\ep,\delta}(\vecf^\prime)
	&:=
	\int_{|s_4-s_3|\ep \leq |t_2| \leq |s_4-s_3|\delta}
	\int_{\mathbb{R} / \mathcal{L}\mathbb{Z}}
	G(\tilde{\theta})^p
	\frac{ \| \Delta_{t_1}^{t_1+t_2} \vecf^\prime \|_{\mathbb{R}^d}^{2p} }
	{ \Psi(|t_2|)^p }
	\frac 1 {|t_2|}
	dt_1dt_2
\\
	I^2_{\ep,\delta}(\vecf^\prime)
	&:=
	\int_{|s_4-s_3|\delta \leq |t_2| \leq |s_4-s_3|\mathcal{L}/2}
	\int_{\mathbb{R} / \mathcal{L}\mathbb{Z}}
	G(\tilde{\theta})^p
	\frac{ \| \Delta_{t_1}^{t_1+t_2} \vecf^\prime \|_{\mathbb{R}^d}^{2p} }
	{ \Psi(|t_2|)^p }
	\frac 1 {|t_2|}
	dt_1dt_2.
\end{align*}
If
$|s_4-s_3|\ep \leq |t_2| \leq |s_4-s_3|\delta$,
we have
$G(\tilde{\theta}) \leq K+\eta$
because
$0 < \tilde{\theta} \leq \delta$.
Hence,
we get
\[
	I^1_{\ep,\delta} (\vecf^\prime) \leq (K+ \eta)^p [\vecf^\prime]_{\Psi , 2p}^{2p}.
\]
If
$|s_4-s_3|\delta \leq |t_2| \leq |s_4-s_3|\mathcal{L}/2$,
then we have
$C_{\text b}^{-1} \delta \leq \tilde{\theta} \leq \mathcal{L}/2$.
Hence,
we get
\[
	I^2_{\ep,\delta}(\vecf^\prime) \leq G_\delta^p [\vecf^\prime]_{\Psi , 2p}^{2p},
\]
where
$G_\delta := \max_{x \in [C_{\text b}^{-1}\delta , \mathcal{L}/2] } G(x)$.
By (A3),
we obtain
\[
	\mathcal{E}^{\Phi , p}_\ep (\vecf)
	\leq
	\frac{C_{\text b}^p}{2^p} \{ (K+ \eta)^p + G_\delta^p\} M(C_{\text b}) [\vecf^\prime]_{\Psi , 2p}^{2p} < \infty
\]
for all
$\ep \leq \delta$.
Thus it holds that
$\mathcal{E}^{\Phi , p}(\vecf) < \infty$.

Next,
we assume
$\mathcal{E}^{\Phi , p}(\vecf) < \infty$.
Then,
we have
\begin{align*}
	&\infty > \mathcal{E}^{\Phi , p}(\vecf)
	=
	\int_{\mathbb{R} / \mathcal{L}\mathbb{Z}} \int_{-\mathcal{L}/2}^{\mathcal{L}/2}
	( g_{|s_2|} (\| \Delta_{s_1}^{s_1+s_2} \vecf \|_{\mathbb{R}^d} )^p ds_2ds_1
\\
	&=
	\int_{\mathbb{R} / \mathcal{L}\mathbb{Z}} \int_{-\mathcal{L}/2}^{\mathcal{L}/2}
	\{ -g_{|s_2|}^\prime (\theta)
	(|s_2| - \| \Delta_{s_1}^{s_1+s_2} \vecf \|_{\mathbb{R}^d})
	\}^p ds_2ds_1
\\
	&\geq
	\frac 1 {4^p}
	\int_{\mathbb{R} / \mathcal{L}\mathbb{Z}} \int_{-\mathcal{L}/2}^{\mathcal{L}/2}
	\left(
	G(\theta)
	\int_0^1 \int_0^1
	\| \Delta_{s_1+s_3s_2}^{s_1+s_4s_2} \vecf^\prime \|_{\mathbb{R}^d}
	ds_4ds_3
	\right)^p
	ds_2ds_1
\\
	&\geq
	\frac{ (K+\tilde{G}_\delta)^p }{ 4^p }
	\tilde{\mathcal{E}}^{\Phi , p} (\vecf^\prime),
\end{align*}
where
$\tilde{G}_\delta := \min_{x \in [C_{\text b}^{-1}\delta , \mathcal{L}/2] } G(x)$.
Hence,
we get inequality (\ref{thm:bdd:ineq}) because it holds that
\[
	\| \vecf^\prime \|_{W^{\Psi , 2p}}
	\leq
	C_{\text g} \left( \| \vecf^\prime \|_{L^{2p}} + \tilde{\mathcal{E}}^{\Phi , p} (\vecf^\prime) \right)
	\leq
	C_{\text g} \left( \| \vecf^\prime \|_{L^{2p}} + \mathcal{E}^{\Phi , p} (\vecf) \right)
\]
by Lemma \ref{lem:bdd},
where
$C_{\text g}$
is a positive constant depending only
$p$,
$\mathcal{L}$,
and
$\Phi$
and which may not be the same in each case.
Therefore,
we obtain
$\vecf \in W^{1+ \Psi , 2p}(\mathbb{R} / \mathcal{L} \mathbb{Z} , \mathbb{R}^d)$.
\end{proof}

%
%

\section{A discretization of the O'Hara energies}\label{dis}

Although minimizers of the O'Hara energies were studied,
it is difficult to calculate the O'Hara energies directly,
and as a result,
it is not easy to evaluate well-balancedness.
In
\cite{KK93},
Kim and Kusner considered a discretization of the M\"{o}bius energy and numerically calculated values of M\"{o}bius energy of torus knots.
Scholtes
\cite{S14}
discussed convergence of Kim-Kusner's discretization,
but he did not use the M\"{o}bius invariance.
Therefore,
we expect that we can consider convergence of a discretization of not only M\"{o}bius energy but also the other O'Hara energies
$\mathcal{E}^{\alpha , p}$.
Actually,
in
\cite{K18},
a discretization of the O'Hara energies was defined, and convergence of this discretization were discussed.
In this section,
we mention the result of
\cite{K18}
and give some examples of numerical calculations of this discretization.

From now on,
we write
$\sigma = (\alpha p-1)/(2p)$
for
$\alpha \in (0,\infty)$
and
$p \in [1,\infty)$
with
$2 \leq \alpha p < 2p+1$.
Also,
we call an \textit{$n$-gon} a polygon with $n$ edges.
For a given regular curve
$\vecf$,
we say that a polygon
$\vecp$
is \textit{inscribed} in
$\vecf$
if
$\vecp$
satisfies
\begin{enumerate}
\renewcommand{\labelenumi}{(\roman{enumi})}
\item
	the number of vertices is finite.
\item
	the set of vertices is
	$\{ \vecf(b_1) , \ldots , \vecf(b_n) \}$
	with
	$b_1 < \cdots < b_n (< b_1+ \mathcal{L})$,
\item
	the $k$-th edge is the segment jointing
	$\vecf(b_k)$
	and
	$\vecf(b_{k+1})$,
	where we interpret
	$b_{n+1} = b_1$.
\end{enumerate}
For
$\alpha$,
$p \in (0,\infty)$,
our discretization of the O'Hara energies is defined by
\begin{multline*}
	\mathcal{E}^{\alpha ,p}_n (\vecp_n)
	:=
	\sum_{\substack{i,j=1\\i\ne j}}^n
	\left(
	\frac 1 { \| \vecp_n(a_j) - \vecp_n(a_i) \|_{\mathbb{R}^d}^\alpha }
	-
	\frac 1 { \mathscr{D}(\vecp_n(a_i) , \vecp_n(a_j))^\alpha }
	\right)^p
\\
	\times
	\| \vecp_n(a_{i+1}) - \vecp_n(a_i) \|_{\mathbb{R}^d}
	\| \vecp_n(a_{j+1}) - \vecp_n(a_j) \|_{\mathbb{R}^d},
\end{multline*}
where
$\vecp_n : \mathbb{R} / \mathcal{L}\mathbb{Z} \to \mathbb{R}^d$
is an $n$-gon parametrized by arc-length whose total length is
$\mathcal{L}_n$,
and
$a_j$
is the value of arc-length parameter corresponding to vertex of
$\vecp_n$
and we assume
$0 \leq a_1 < \cdots < a_n < \mathcal{L}_n$
$(\operatorname{mod}\mathcal{L}_n)$.

The following theorem obtained in
\cite{K18}
gives us convergence of our discretization as
$n \to \infty$
and the rule of convergence.

\begin{theorem}[Approximation of the O'Hara energies by inscribed polygons,
\cite{K18}]\label{thm:convergence}
Assume that
$\alpha \in (0, \infty)$
and
$p \in [1, \infty)$
satisfy
$2 \leq \alpha p < 2p+1$.
Let
$\vecf \in C^{1,1} (\mathbb{R} / \mathcal{L} \mathbb{Z} , \mathbb{R}^d)$
be a function which image is a closed curve parametrized by arc-length embedded in
$\mathbb{R}^d$,
where
$\mathcal{L}$
is the length of
$\vecf$.
Let
$c$,
$\bar{c} >0$,
and set
$K = \| \vecf^{\prime \prime} \|_{L^\infty (\mathbb{R} / \mathcal{L} \mathbb{Z} , \mathbb{R}^d)}$.

In addition,
for
$n \in \mathbb{N}$,
let
$\{ b_k \}_{k=1}^n$
be a partition of
$\mathbb{R} / \mathcal{L}\mathbb{Z}$
satisfying
\[
	\frac{c \mathcal{L}} n \leq \| \vecf(b_{k+1}) - \vecf(b_k) \|_{\mathbb{R}^d}
	\leq \frac{\bar{c}\mathcal{L}} n,
\]
and let
$\vecp_n$
be the inscribed polygon in
$\vecf$
with vertices
$\vecf(b_1)$,
$\ldots$,
$\vecf(b_n)$.
Then,
if the number
$n$
of points of the division is sufficiently large,
there exists
$C>0$
such that
\[
	| \mathcal{E}^{\alpha, p} (\vecf) - \mathcal{E}^{\alpha, p}_n (\vecp_n) |
	\leq
	C \frac 1 { n^{2p-\alpha p+1} }.
\]
Moreover,
if
$\vecf \in W^{1+ \sigma ,2p} (\mathbb{R} / \mathcal{L}\mathbb{Z} , \mathbb{R}^d )$,
we have
\[
	\lim_{n \to \infty} \mathcal{E}^{\alpha, p}_n (\vecp_n)
	=
	\mathcal{E}^{\alpha, p} (\vecf).
\]
\end{theorem}

Moreover,
we obtained the
$\Gamma$-convergence of
$\mathcal{E}^{\alpha , p}_n$
to
$\mathcal{E}^{\alpha , p}$
as
$n \to \infty$
in
\cite{K18}.

\begin{theorem}[$\Gamma$-convergence of $\mathcal{E}_n^{\alpha , p}$,
\cite{K18}]
Let
$\alpha \in (0,\infty)$
and
$p \in [1,\infty)$
with
$2 \leq \alpha p < 2p+1$.
Then,
it holds that
$\mathcal{E}_n^{\alpha , p}$
$\Gamma$-converges to
$\mathcal{E}^{\alpha , p}$
on the metric space
$X$
given by
\[
	X := \left(
	\left( \mathcal{C}(\mathcal{K}) \cap C^1(\mathbb{R} / \mathcal{L}\mathbb{Z} , \mathbb{R}^d) \right)
	\cup \bigcup_{n \in \mathbb{N}} \mathcal{P}_n(\mathcal{K}), d_X \right).
\]
Here,
$\mathcal{K}$
is a tame knot class,
$\mathcal{C}(\mathcal{K})$
is the set of simply closed curves of length $1$ belonging to
$\mathcal{K}$,
$\mathcal{P}_n (\mathcal{K})$
is the set of equilateral $n$-gons with total length $1$ belonging to
$\mathcal{K}$,
and the metric
$d_X$
is such that,
for some constants
$C_1$,
$C_2 >0$,
we have
\[
	C_1 \| \vecf - \vecg \|_{L^1} \leq d_X (\vecf , \vecg)
	\leq C_2 \| \vecf - \vecg \|_{W^{1, \infty}}	
\]
for
$\vecf$,
$\vecg \in X$.
\end{theorem}

By the property of
$\Gamma$-convergence
(e.g.\ in \cite{DM93}),
the minimum values of
$\mathcal{E}_n^{\alpha ,p}$
converge to that of
$\mathcal{E}^{\alpha , p}$
which is attained by a right circle
(cf.\ \cite{ACFGH03}).
Thus,
it is natural to consider minimizers of
$\mathcal{E}_n^{\alpha , p}$.
In
\cite{K18},
we can completely characterize minimizers of a generalized discrete functional defined by
\begin{multline*}
	\mathcal{E}^F_n (\vecp_n)
	:=
	\sum_{\substack{i,j=1\\ i\ne j}}^n
	F( \| \vecp_n(a_j) - \vecp_n(a_i) \|_{\mathbb{R}^d} , \mathscr{D}(\vecp_n(a_i) , \vecp_n(a_j) ) )
\\
	\times
	\| \vecp_n(a_{i+1}) - \vecp_n(a_i) \|_{\mathbb{R}^d}
	\| \vecp_n(a_{j+1}) - \vecp_n(a_j) \|_{\mathbb{R}^d},
\end{multline*}
where
$\vecp_n$
is an $n$-gon with total length $1$,
and
$F$
is a real-valued function on
$\Omega := \{ (x,y) \in \mathbb{R}^2 \,|\, 0<x \leq y \}$.

\begin{theorem}[Minimizers of $\mathcal{E}^F_n$,
\cite{K18}]
Assume that
$F : \Omega \to \mathbb{R}$
is such that if we set
$g_y(u) = F(\sqrt{u} , y)$
for
$u \in (0, y^2]$
and
$y \in (0,1/2)$,
then
$g_y$ is decreasing and convex.
Moreover,
for
$0<a<b$,
set
$[a]_b := \min \{ a, b-a \}$.
Then,
if
$\vecp_n$
is an {\rm equilateral} polygon,
we have
\[
	\mathcal{E}^F_n(\vecp_n)
	\geq
	\frac 1 n
	\sum_{k=1}^{n-1} F \left(
	\frac{ \sin ([k]_n\pi / n) }{ n \sin(\pi / n) } , \mathscr{D}(\vecp_n(a_k) , \vecp_n(a_0))
	\right)
\]
and the minimizers of
$\mathcal{E}^F_n$
are regular $n$-gons.
\end{theorem}

If we set
$F(x,y) := ( x^{-\alpha} - y^{-\alpha} )^p$,
then
$\mathcal{E}^F_n$
corresponds to
$\mathcal{E}^{\alpha , p}_n$.
Then,
we get the following corollary.

\begin{corollary}\label{coro:min}
Let
$\alpha \in (0, \infty)$
and
$p \in [1,\infty)$.
Then,
minimizers of
$\mathcal{E}^{\alpha , p}_n$
in the set of {\rm equilateral} $n$-gons are regular polygons.
In particular,
a regular polygon with $n$ edges is the only minimizer up to congruent transformations and similar transformations.
\end{corollary}

\begin{remark}
There do not exist minimizers of
$\mathcal{E}^{\alpha , p}_n$
in the set of \textit{all} $n$-gons which are not necessarily equilateral.
Indeed,
considering an $(n-1)$-gon as a degenerate $n$-gon,
we have
\[
	0 \leq
	\inf_{n\text{-gon}} \mathcal{E}^{\alpha , p}_n
	\leq
	\inf_{(n-1)\text{-gon}} \mathcal{E}^{\alpha , p}_{n-1}
	\leq
	\cdots
	\leq
	\inf_{3\text{-gon}} \mathcal{E}^{\alpha , p}_3,
\]
and because
$\mathcal{E}^{\alpha , p}_3 (\vecp_3) = 0$
for all $3$-gons
$\vecp_3$,
we obtain
\[
	\inf_{n\text{-gon}} \mathcal{E}^{\alpha , p}_n = 0.
\]
\end{remark}

Next,
we show some examples of numerical experiments.
Let
$\vecg_n$
be a regular $n$-gon.
By the property of
$\Gamma$-convergence and Corollary \ref{coro:min},
we have
\[
	\inf \mathcal{E}^{\alpha , p} = \lim_{n \to \infty} \mathcal{E}^{\alpha , p}_n (\vecg_n),
\]
where
$2 \leq \alpha p < 2p+1$,
and the infimum in the left-hand side is taken over the space of all embedded curves in
$\mathbb{R}^d$.
Therefore,
considering
\cite{ACFGH03},
we can calculate the O'Hara energy of a right circle numerically by increasing the number of vertices $n$ in
$\mathcal{E}^{\alpha , p}_n(\vecg_n)$.
Moreover,
we calculate energies
$\mathcal{L}(\vecg_n)^{\alpha p -2} \mathcal{E}^{\alpha , p}_n(\vecg_n)$,
where
$\mathcal{L}(\vecg_n)^{\alpha p -2}$
is the total length of
$\vecg_n$,
because the factor
$\mathcal{L}(\vecg_n)^{\alpha p -2}$
makes these energies scale invariant.
Note that in
\cite{IN18},
the values of the O'Hara energy
$\mathcal{E}^{\alpha ,1}$
$(2 \leq \alpha < 3)$
of a right circle
$\vecf_0$
are obtained and expressed by
\[
	\mathcal{E}^{\alpha , 1}(\vecf_0)
	=
	\frac 1 { (\alpha-1) \mathcal{L}(\vecf_0)^{\alpha -2} }
	\left\{
	\frac{ (\alpha -2) \pi^{\alpha -1/2} \Gamma((3-\alpha)/2) }{ \Gamma((4-\alpha)/2) } + 2^\alpha
	\right\}.
\]
Here,
we compare
$\mathcal{L}(\vecg_n)^{\alpha - 2} \mathcal{E}^{\alpha , 1}_n(\vecg_n)$
with
$\mathcal{L}(\vecf_0)^{\alpha - 2} \mathcal{E}^{\alpha , 1}(\vecf_0)$,
and we tabulate the result of numerical calculation when
$\alpha = 2,2.1,2.3,2.5,2.7,2.9$
in Table \ref{table:cal}.
\begin{table}[ht]
\begin{center}
\fontsize{7pt}{9pt}\selectfont
\begin{tabular}{l|c||c|c|c|c|c|c}
 & Number of & \multicolumn{6}{|c}{$\alpha$}\\ \cline{3-8}
 & vertices $n$ & $2$ & $2.1$ & $2.3$ & $2.5$ & $2.7$ & $2.9$ \\ \hline \hline
 & $4$ & $1$ & $1.147365$ & $1.500936$ & $1.949372$ & $2.516555$ & $3.232177$ \\ \cline{2-8}
 & $8$ & $2.325253$ & $2.739102$ & $3.780728$ & $5.187945$ & $7.085586$ & $9.640817$ \\ \cline{2-8}
 & $16$ & $3.134412$ & $3.754475$ & $5.372714$ & $7.672833$ & $10.95137$ & $15.64031$ \\ \cline{2-8}
 & $32$ & $3.562332$ & $4.320470$ & $6.363289$ & $9.408493$ & $13.99728$ & $20.99456$ \\ \cline{2-8}
 & $64$ & $3.780229$ & $4.626457$ & $6.969742$ & $10.61781$ & $16.42130$ & $25.87401$ \\ \cline{2-8}
 & $128$ & $3.889916$ & $4.790718$ & $7.341313$ & $11.46626$ & $18.37252$ & $30.38526$ \\ \cline{2-8}
 & $256$ & $3.944913$ & $4.878765$ & $7.569466$ & $12.06415$ & $19.95194$ & $34.58121$ \\ \cline{2-8}
 & $512$ & $3.972446$ & $4.925946$ & $7.709746$ & $12.48634$ & $21.23325$ & $38.49223$ \\ \cline{2-8}
 & $1024$ & $3.986220$ & $4.951228$ & $7.796054$ & $12.78472$ & $22.27356$ & $42.14019$ \\ \cline{2-8}
 & $2048$ & $3.993109$ & $4.964776$ & $7.849171$ & $12.99567$ & $23.11844$ & $45.54354$ \\ \cline{2-8}
D & $4096$ & $3.996555$ & $4.972036$ & $7.881865$ & $13.14482$ & $23.80467$ & $48.71889$ \\ \cline{2-8}
 & $8192$ & $3.998277$ & $4.975926$ & $7.901990$ & $13.25028$ & $24.36205$ & $51.68157$ \\ \cline{2-8}
 & $16384$ & $3.999139$ & $4.978011$ & $7.914378$ & $13.32485$ & $24.81478$ & $54.44584$ \\ \cline{2-8}
 & $32768$ & $3.999569$ & $4.979129$ & $7.922004$ & $13.37758$ & $25.18251$ & $57.02499$ \\ \cline{2-8}
 & $65536$ & $3.999785$ & $4.979727$ & $7.926698$ & $13.41487$ & $25.48120$ & $59.43143$ \\ \cline{2-8}
 & $131072$ & $3.999892$ & $4.980048$ & $7.929588$ & $13.44124$ & $25.72381$ & $61.67671$ \\ \cline{2-8}
 & $262144$ & $3.999946$ & $4.980220$ & $7.931366$ & $13.45988$ & $25.92087$ & $63.77161$ \\ \cline{2-8}
 & $524288$ & $3.999973$ & $4.980312$ & $7.932461$ & $13.47306$ & $26.08094$ & $65.72639$ \\ \cline{2-8}
 & $1048576$ & $3.999987$ & $4.980362$ & $7.933135$ & $13.48238$ & $26.21093$ & $67.55013$ \\ \cline{2-8}
 & $2097152$ & $4.000004$ & $4.980401$ & $7.933568$ & $13.48900$ & $26.31651$ & $69.25143$ \\ \cline{2-8}
 & $4194304$ & $3.999997$ & $4.980402$ & $7.933807$ & $13.49362$ & $26.40257$ & $70.84417$ \\ \hline \hline
\multicolumn{2}{c||}{ Analytic values } & $4$ & $4.980419$ & $7.934215$ & $13.50489$ & $26.77342$ & $92.95965$ \\ \hline \hline
\multicolumn{2}{c||}{ D$/$A } & $0.999999$ & $0.999997$ & $0.999949$ & $0.999166$ & $0.986148$ & $0.762096$
\end{tabular}
\normalsize
\caption{Numerical calculation of
$\mathcal{L}(\vecf_0)^{\alpha -2}\mathcal{E}^{\alpha , 1}(\vecf_0)$
when
$2 \leq \alpha < 3$
(D:
Values of discretization,
D$/$A:
Divisions of value of discretization when
$n=4194304$
by analytic value)}
\label{table:cal}
\end{center}
\end{table}
It follows from Theorem \ref{thm:convergence} that the convergence becomes slow,
when
$\alpha$
approaches to $3$.
We can see this fact from Table \ref{table:cal}.
Moreover,
we investigate the behavior of
\[
	e_\alpha (n) :=
	n^{\alpha -2}
	\left|
	\mathcal{L}(\vecf_0)^{\alpha - 2} \mathcal{E}^{\alpha , 1}(\vecf_0)
	-
	\mathcal{L}(\vecg_n)^{\alpha - 2} \mathcal{E}^{\alpha , 1}_n(\vecg_n)
	\right|
\]
when number of vertices $n$ increases,
where
$2 \leq \alpha < 3$.
We expect that
$e_\alpha (n)$
converges to a constant if the order of convergence in Theorem \ref{thm:convergence} is optimal,
and we can see that this conjecture seems to be true in Figure \ref{figure:error}.

Now,
we show some interesting examples of
$\mathcal{E}_n^{\alpha , p}(\vecg_n)$
when the number of vertices $n$ is not so large.
As we can see in Figure \ref{figure:2--30},
$\mathcal{L}(\vecg_{2^k})^{58}\mathcal{E}^{2,30}_{2^k}(\vecg_{2^k})$
for
$k \in \mathbb{N}$
takes the maximum value at
$k=4$,
and the larger the value that
$p$
takes,
the larger the maximum value is.
Therefore,
we show a figure of
$\mathcal{E}^{2,30}_n (\vecg_n)$
for
$n \geq 100$
in Figure \ref{figure:2--30low}.
Note that
$\mathcal{L}(\vecg_{2\ell +1})^{58} \mathcal{E}^{2,30}_{2\ell +1}(\vecg_{2\ell +1})$
for
$\ell \geq 2$
is monotonically increasing.
However,
$\mathcal{L}(\vecg_{2\ell})^{58} \mathcal{E}^{2,30}_{2\ell}(\vecg_{2\ell})$
for
$\ell \geq 2$
takes the maximum at
$\ell =10$
($n=20$)
and is decreasing to the value of
$\mathcal{L}(\vecf_0)^{58} \mathcal{E}^{2,30}(\vecf_0)$
when
$\ell \geq 10$.
The cause of this phenomena we think is as follows:
when
$n$
is much less than $20$,
the discrete energy is a summation which consists of a small number of terms with large value.
On the other hand,
when
$n$
is much larger than $20$,
the discrete energy is a summation which consists of a large number of terms with small value.
If
$n$
is around $20$,
then the number of terms and the size of each term might make the energy large.
This phenomena will be remarkable when
$p$
becomes large.
To the auther,
the reason seems to be as follows:
when
$p$
is large,
the difference of the size of the terms becomes bigger.
Moreover,
we observe from Figure \ref{figure:2--30low} that the energy with even $n$ is larger than that with odd $n$.
The energy density becomes large when the difference between the intrinsic distance and the extrinsic distance is large.
The difference maximizes when two points are antipodal,
which is a situation that occurs only when $n$ is even.

\newpage
\begin{figure}[ht]
\begin{tabular}{ccc}
\begin{minipage}[t]{0.3\hsize}
\begin{center}
\includegraphics[height=2.5cm,width=3.9cm]{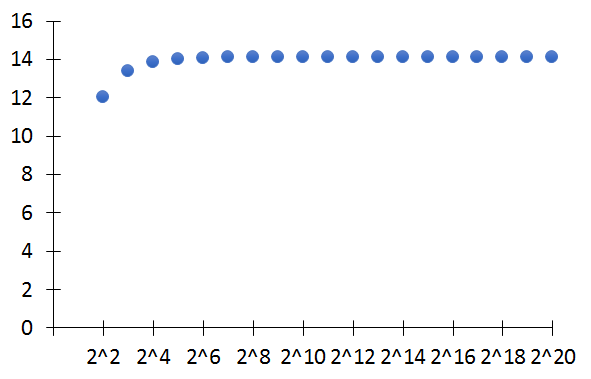}
\end{center}
\begin{center}
$\alpha = 2$
\end{center}
\end{minipage} &
\begin{minipage}[t]{0.3\hsize}
\begin{center}
\includegraphics[height=2.5cm,width=3.9cm]{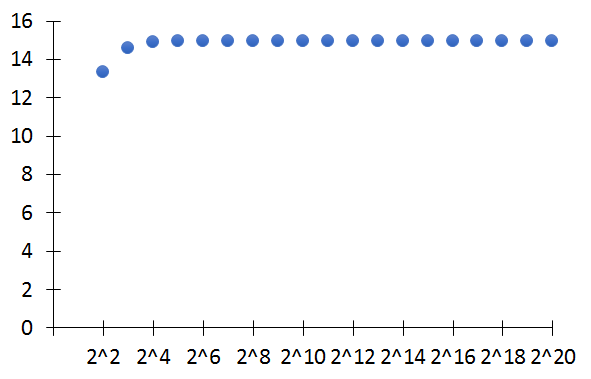}
\end{center}
\begin{center}
$\alpha = 2.1$
\end{center}
\end{minipage} &
\begin{minipage}[t]{0.3\hsize}
\begin{center}
\includegraphics[height=2.5cm,width=3.9cm]{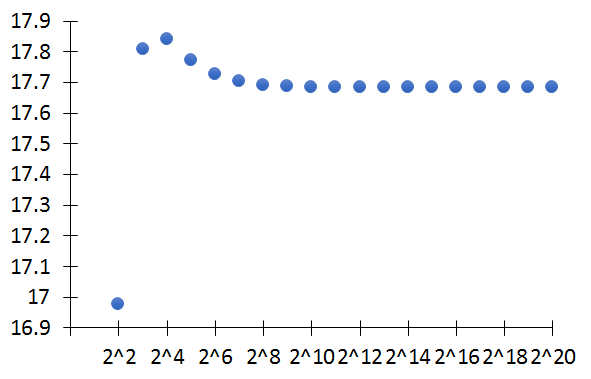}
\end{center}
\begin{center}
$\alpha = 2.3$
\end{center}
\end{minipage}
\end{tabular}
\begin{tabular}{ccc}
\begin{minipage}[t]{0.3\hsize}
\begin{center}
\includegraphics[height=2.5cm,width=3.9cm]{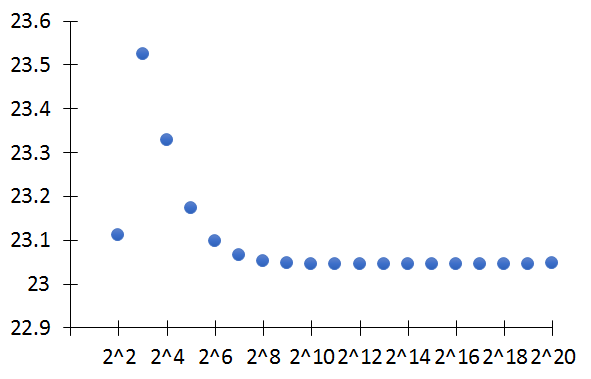}
\end{center}
\begin{center}
$\alpha = 2.5$
\end{center}
\end{minipage} &
\begin{minipage}[t]{0.3\hsize}
\begin{center}
\includegraphics[height=2.5cm,width=3.9cm]{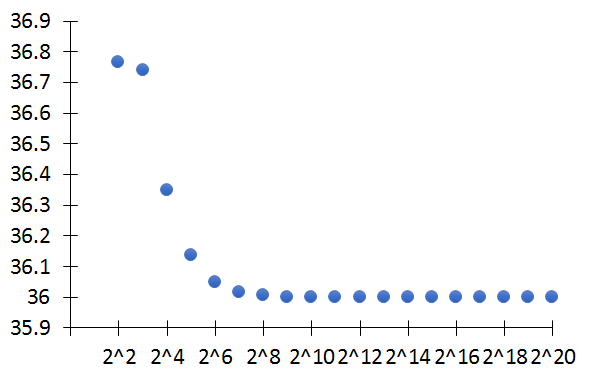}
\end{center}
\begin{center}
$\alpha = 2.7$
\end{center}
\end{minipage} &
\begin{minipage}[t]{0.3\hsize}
\begin{center}
\includegraphics[height=2.5cm,width=3.9cm]{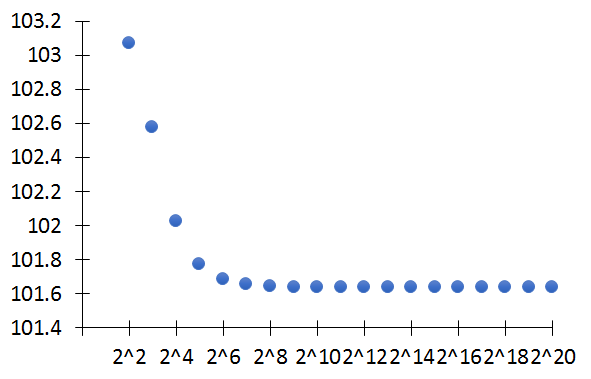}
\end{center}
\begin{center}
$\alpha = 2.9$
\end{center}
\end{minipage}
\end{tabular}
\caption{Graphs of $e_\alpha (n)$
(The vertical and horizontal axes show values of
$e_\alpha(n)$
and numbers of vertices
$n=2^k$
($k=2,3, \cdots , 20$),
respectively)}\label{figure:error}
\vspace{1cm}
\begin{tabular}{cc}
\begin{minipage}[t]{0.45\hsize}
\begin{center}
\includegraphics[height=3.5cm,width=5.5cm]{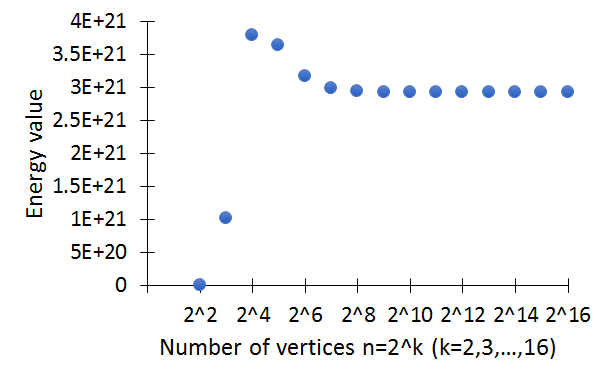}
\end{center}
\caption{Values of
$\mathcal{E}^{2,30}_{2^k}(\vecg_{2^k})$}
\label{figure:2--30}
\end{minipage} &
\begin{minipage}[t]{0.45\hsize}
\begin{center}
\includegraphics[height=3.5cm,width=6cm]{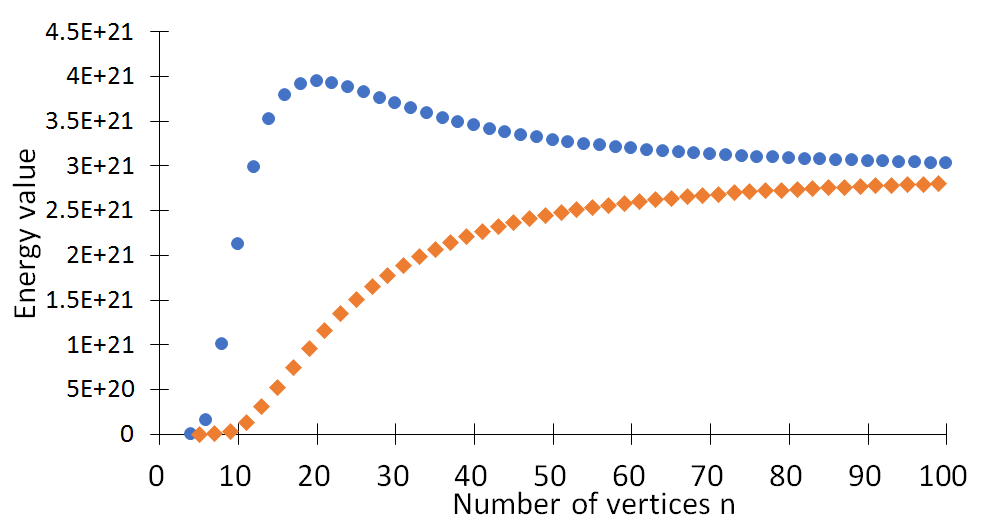}
\end{center}
\caption{Values of
$\mathcal{E}^{2,30}_n(\vecg_n)$
when
$n \leq 100$
(Blue, round points and orange, diamond points show values when $n$ is even and odd,
respectively)}
\label{figure:2--30low}
\end{minipage}
\end{tabular}
\end{figure}
\newpage

%
%

\section{Conclusions and future work}
In Section \ref{gene},
we considered the generalized O'Hara energy
$\mathcal{E}^{\Phi , p}$
and characterized the finiteness of these energies by using the generalized Sobolev-Slobodeckii space
$W^{1+\Psi , 2p}$.
However,
several problems concerning
$\mathcal{E}^{\Phi , p}$
remain open,
e.g.,\ conditions which
$\mathcal{E}^{\Phi , p}$
is the knot energy
(with regard to the definition, see
\cite{O92}),
and the existence of minimizers of
$\mathcal{E}^{\Phi , p}$
in a given knot type.
In Section \ref{dis},
we discussed a discretization defined in
\cite{K18}
of not only the M\"{o}bius energy but also the O'Hara energy
and numerically calculated the energy values of a right circle.
Several researchers have considered numerical calculations of
$\mathcal{E}^{\alpha , 1}$
of not only a circle but also various knots.
However,
numerical calculation of
$\mathcal{E}^{\alpha , p}$
($p>1$),
except a right circle,
is yet to be carried out;
this will be addressed om forthcoming work of the author.

\paragraph{Acknowledgments}
The author is grateful to Professor Takeyuki Nagasawa for his direction and many useful advices and remarks. Moreover,
the author would like to thank Professor Richard Neal Bez for English language editing and mathematical comments.
%
%


\begin{thebibliography}{99}

\bibitem{ACFGH03}
	\newblock A. Abrams, J. Cantarella, J. H. G. Fu, M. Ghomi, and R. Howard,
	\newblock Circles minimize most knot energies,
	\newblock \emph{Topology}, \textbf{42}(2) (2003), 381--394.

\bibitem{B12}
     \newblock S. Blatt,
     \newblock Boundedness and regularizing effects of O'Hara's knot energies,
     \newblock \emph{J. Knot Theory ramifications}, \textbf{21} (2012), 1250010, 9 pp.

\bibitem{DM93}
	\newblock G. Dal Maso,
	\newblock \emph{An introduction to $\Gamma$-convergence},
	\newblock Progress in Nonlinear Diffrential Equations and their Applications, vol. 8, Birkh\"{a}user Boston, Boston, MA, 1993.

\bibitem{FHW94}
	\newblock M. H. Freedman, Z.-X. He, and Z. Wang,
	\newblock M\"{o}bius energy of knots and unknots,
	\newblock \emph{Ann. of Math.}, \textbf{139} (1994), 1--50.

\bibitem{IN18}
	\newblock A. Ishizeki and T. Nagasawa,
	\newblock Decomposition of generalizaed O'Hara's energies,
	\newblock arXiv:1904.06812.

\bibitem{K18}
	\newblock S. Kawakami,
	\newblock A discretization of O'Hara's knot energy and its convergence,
	\newblock arXiv:1908.11172

\bibitem{KK93}
	\newblock D. Kim and R. Kusner,
	\newblock Torus knots extremizing the M\"{o}bius energy,
	\newblock \emph{Experiment. Math.}, \textbf{2}(1) (1993), 1--9. 

\bibitem{M06}
	\newblock S. Miyajima,
	\newblock \emph{Introduction to Sobolev space and its application},
	\newblock Kyoritsu Shuppan, Tokyo, 2006, in Japanese.

\bibitem{O91}
	\newblock J. O'Hara,
	\newblock Energy of a knot,
	\newblock \emph{Topology}, \textbf{30}(2) (1991), 241--247.

\bibitem{O92}
	\newblock J. O'Hara,
	\newblock Family of energy functionals of knots,
	\newblock \emph{Topology Appl.}, \textbf{48} (1992), 147--161.

\bibitem{O94}
	\newblock J. O'Hara,
	\newblock Energy functionals of knots II,
	\newblock \emph{Topology Appl.}, \textbf{48} (1994), 45--61.

\bibitem{RS06}
	\newblock E. J. Rawdon and J. K. Simon,
	\newblock Polygonal approximation and energy of smooth knots.
	\newblock \emph{J. Knot Theory Ramifications}, \textbf{15}(4) (2006), 429--451.

\bibitem{S14}
	\newblock S. Scholtes,
	\newblock Discrete M\"{o}bius energy,
	\newblock \emph{J. Knot Theory Ramifications}, \textbf{23} (2014), 1450045, 16 pp.

\bibitem{Si94}
	\newblock J. K. Simon,
	\newblock Energy functions for polygonal knots,
	\newblock \emph{J. Knot Theory Ramifications}, \textbf{3}(3) (1994), 299--320.

\end{thebibliography}
\end{document}